\documentclass{amsart}
\usepackage{hyperref}
\hypersetup{colorlinks = true, linkcolor = black, citecolor = black}

\DeclareMathOperator{\vol}{vol}
\DeclareMathOperator{\cs}{cs}
\DeclareMathOperator{\sn}{sn}
\DeclareMathOperator{\tr}{tr}

\newcommand{\m}{{\scriptscriptstyle{-}}}
\newcommand{\p}{{\scriptscriptstyle{+}}}
\newcommand{\norm}[1]{\lVert{#1}\rVert}

\let\phi\varphi

\newtheorem{theorem}{Theorem}[section]

\newtheorem{lemma}[theorem]{Lemma}
\newtheorem{corollary}[theorem]{Corollary}

\theoremstyle{definition}

\theoremstyle{remark}
\newtheorem{remark}[theorem]{Remark}

\numberwithin{equation}{section}

\begin{document}

\title[Volume estimates for tubes using integral curvature bounds]{Volume estimates for tubes around submanifolds using integral curvature bounds}

\author{Yousef K. Chahine}
\address{University of California, Santa Barbara}
\curraddr{}
\email{ychahine@math.ucsb.edu}
\thanks{}


\date{}

\dedicatory{}

\begin{abstract}
We generalize an inequality of E. Heintze and H. Karcher \cite{HK78} for the volume of tubes around minimal submanifolds to an inequality based on integral bounds for $k$-Ricci curvature.  Even in the case of a pointwise bound, this generalizes the classical inequality by replacing a sectional curvature bound with a $k$-Ricci bound.  This work is motivated by the estimates of Petersen-Shteingold-Wei for the volume of tubes around a geodesic \cite{PSW97} and generalizes their result.  Using similar ideas we also prove a Hessian comparison theorem for $k$-Ricci curvature which generalizes the usual Hessian and Laplacian comparison for distance functions from a point and give several applications.
\end{abstract}

\maketitle

\section{Introduction}\label{sec:intro}

The geodesic tube of radius $r$ around a closed submanifold $\Sigma^m$ of a Riemannian manifold $M^n$, denoted $T(\Sigma,r)$, is the set of all points whose distance to $\Sigma$ is at most $r$.
In this paper, we give upper bounds for the volume of $T(\Sigma,r)$ based on $L^p$ norms of the negative part of the $k$-Ricci curvature of $M$.  For $p=\infty$, we prove that the well-known estimate of E. Heintze and H. Karcher based on pointwise sectional curvature bounds requires only $k$-Ricci bounds (Theorem \ref{thm:pointwiseestimate}).  The main result is the case $p<\infty$, where we give the first estimates for the volume of tubes around submanifolds of general codimension using integral curvature bounds (Theorem \ref{thm:integralestimate}).

The $k$-Ricci curvature interpolates between sectional curvature and Ricci curvature by taking an average of sectional curvatures over a $k$-dimensional subspace of the tangent space.  Specifically, given a unit vector $u$ tangent to $M$ and $k$-dimensional subspace $\mathcal V$ of the tangent space orthogonal to $u$ the $k$-Ricci curvature of $(u,\mathcal V)$ is defined by
$$Ric_k(u,\mathcal V)=\sum_{i=1}^k \langle R(e_i,u)u,e_i\rangle$$
where $e_1,...,e_k$ form an orthonormal basis of $\mathcal V$.  Notice that $Ric_{n-1}$ is equivalent to the Ricci curvature and $Ric_1$ is equivalent to sectional curvature.  We say that a manifold has $k$-Ricci curvature bounded below by $kH$ for some constant $H$ if $Ric_k(u,\mathcal V)\geq kH$ for all unit vectors $u\in TM$ and $k$-dimensional subspaces $\mathcal V\perp u$.

The earliest global results using $k$-Ricci lower bounds as a partial positivity condition for curvature were obtained by Wu \cite{W87}, Shen \cite{S93}, and Shen-Wei \cite{SW93}, though the relationship between $k$-Ricci curvature and volume had been considered previously by Bishop and Crittenden \cite[p. 253]{BC64}.
A signficant literature has since developed which bridges a gap between the global results based on sectional curvature bounds and those based on Ricci curvature bounds \cite{W97, PW03, XY12, GW18, KM18}. 

In a different direction, it has recently been shown that some global results still hold even without pointwise bounds, provided the part of the curvature which violates a pointwise bound is small in an $L^p$ sense \cite{G88,Y92,PW97,PSW97,PS98,A07}.
To make this precise, for a real-valued function $f$ let $f_\p=\max\{f,0\}$ and $f_\m=\max\{-f,0\}$ denote the positive and negative parts of $f$, respectively.  Given a manifold $(M,g)$, let $\rho_k(x)$ denote the minimum of $Ric_k(u,\mathcal V)$ where $u\in T_xM$ is a unit tangent vector at $x$ and $\mathcal V$ is a $k$-dimensional subspace orthogonal to $u$.  For a fixed constant $H$ we may then consider the norms
$$\norm{(\rho_k-H)_\m}_p= \left(\int_M (\rho_k-H)_\m^p \;dvol_g \right)^{1/p}$$
which measure the amount of $k$-Ricci curvature below $H$.

Volume estimates for geodesic balls and tubes around hypersurfaces using bounds on $\norm{(\rho_{n-1}-H)_-}_p$ have been obtained by Gallot \cite{G88}, Yang \cite{Y92}, and Petersen-Wei \cite{PW97}.  The latter established a relative volume comparison using integral curvature bounds and has numerous applications.  Volume estimates for tubes around geodesics using integral curvature bounds were also obtained by Petersen-Shteingold-Wei \cite{PSW97} and used to generalize Cheeger's lemma and the Grove-Petersen finiteness theorem to manifolds with integral curvature bounds.

This last work illustrates the increase in difficulty in the case that $\Sigma$ has arbitrary codimension.  In particular, estimates were needed for certain quadratic invariants of the Hessian of the distance function which were completely new to comparison geometry \cite[Lemma 3.1]{PSW97}.
Our methods are based on the ideas of \cite{PSW97}; however, in that work a number of simplifications were employed specific to the 1-dimensional case which make modification to general codimension nontrivial.  Indeed, if one naively adapts the arguments of \cite{PSW97} the resulting volume estimates require stronger assumptions on the curvature of $M$ and on the second fundamental form of $\Sigma$ than are necessary.

In Section \ref{sec:integralestimate}, we show how these types of estimates generalize to tubes around minimal submanifolds of all dimensions.  Specifically, we prove
\begin{theorem}\label{thm:integralestimate}
Let $M^n$ be a complete Riemannian manifold and let $\Sigma^m\subset M$ be an $m$-dimensional closed minimal submanifold with $0<m<n-1$ and put $k=\min\{m, n-m-1\}$.  If $H\leq 0$ and $p>n-k$ then
$$\vol(T(\Sigma,r))\leq \left(w(r)^{n-m-1} + 2^{p/\alpha}\norm{(\rho_k-H)_-}_p^{\beta p} w(r)^p\right)e^{\kappa r^{2\alpha}}$$
where $\alpha = \frac{n-k-1}{n-k},$ $\beta = \frac{1}{n-m-1}-\frac{1}{p}$, 
\begin{equation*}
w(r) = \left( \tfrac{\alpha}{n-m-1}\right)^\frac{1}{n-k-1}\left(\vol(\mathbb S^{n-m-1})\vol(\Sigma) r^{n-m}\right)^\frac{1}{n-m-1} + \delta \norm{(\rho_k-H)_\m}_p^{1-\beta} r^2,
\end{equation*}
and $\kappa = (\delta|H|)^\alpha/(2\alpha)$ with
$$\delta = 4(n-k-1) + \frac{4}{k}\left(\frac{2p-1}{p-n+k}\right).$$
\end{theorem}

\begin{remark}
As mentioned above, estimates when $\Sigma$ is a point or a hypersurface have already been obtained in \cite{G88,PW97} so we do not repeat this case.
\end{remark}

\begin{remark}
In the case of a pointwise lower bound $Ric_m\geq 0$ (i.e. $\norm{(\rho_m)_\m}_p=0$) with $m\leq n-m-1$ the estimate above reduces to 
$$\vol(T(\Sigma,r))\leq \frac{1}{n-m}\vol(\Sigma)\vol(\mathbb S^{n-m-1}) r^{n-m}.$$
In particular, this shows that the Heintze-Karcher estimate \cite[Corollary 3.3.1]{HK78} holds for tubes around minimal submanifolds assuming only a $k$-Ricci lower bound in place of a sectional curvature lower bound.  In fact, in Theorem \ref{thm:pointwiseestimate} below we show that the Heintze-Karcher estimate holds for tubes around \emph{any} closed submanifold assuming only a pointwise $k$-Ricci lower bound (see also \cite{GM03} for a related volume comparison using pointwise $k$-Ricci bounds).
\end{remark}

Loosely speaking, the estimate of Theorem \ref{thm:integralestimate} shows that it does not matter how the negative part of the curvature concentrates around the submanifold, a uniform estimate holds for all manifolds with $\norm{(\rho_k-H)_\m}_p$ bounded above by a constant as long as $p$ is chosen sufficiently large.  The estimates of Gallot and Petersen-Wei for tubes around hypersurfaces and geodesic balls require $p>n/2$, whereas the estimates of Petersen-Shteingold-Wei for the tubes around a geodesic require $p>n-1$.  Notice that our requirement $p>n-k$ is a natural generalization of both of these conditions as $n-k$ is bounded below by $n/2$.

As a simple consequence of Theorem \ref{thm:integralestimate} we obtain the following uniform lower bound for the volume of closed minimal submanifolds in spaces with integral curvature bounds.

\begin{corollary}\label{cor:cheegerlemma}
Given integers $n$ and $m$ with $n\geq 3$ and $0<m<n-1$, and real numbers $H\leq 0$, $v_0,D>0$ and $p>n-k$ where $k=\min\{m,n-m-1\}$, there exist constants $\epsilon(n,m,p,H,v_0,D)>0$ and $\delta(n,m,p,H,v_0,D)>0$ such that every closed $n$-dimensional Riemannian manifold $M$ satisfying
\begin{align*}
\vol(M)&\geq  v_0,\;\;\; \textup{diam}(M) \leq D ,\;\;\;  \norm{(\rho_k-H)_\m}_p \leq \epsilon
\end{align*}
has the property that all closed $m$-dimensional minimal submanifolds have volume bounded below by $\delta$.
\end{corollary}

\begin{remark}
This should be thought of as a generalization of Cheeger's lemma.  For the case of 1-dimensional minimal submanifolds (closed geodesics) this result was obtained already in \cite[Theorem 1.2]{PSW97}.  The proof follows easily from the observation that our uniform upper bound for the tube around a minimal submanifold approaches 0 as $\vol(\Sigma),\norm{(\rho_k-H)_\m}_p \to 0$.
\end{remark}
Before proving Theorem \ref{thm:integralestimate} we introduce our main ideas by proving a new Hessian comparison for distance functions based on $k$-Ricci curvature bounds which is of independent interest.  Specifically, if $r(x)=d(x,\Sigma)$ is the distance to a closed submanifold $\Sigma$ we prove an upper bound for certain partial traces of the Hessian $\nabla^2r$ given pointwise lower bounds on $Ric_k$.  This Hessian comparison unifies and generalizes a number of distinct Hessian and Laplacian comparisons for the distance function to a point.  Recall that in a space of constant curvature $H$ the eigenvalues of the Hessian of the distance function to a point are given by $\cs_H(r)/\sn_H(r)$ where $\sn_H$ and $\cs_H$ are the generalized trigonometric functions defined in Section \ref{sec:hessiancomparison}.

\begin{theorem}[Hessian Comparison]\label{thm:hessiancomparison}
Let $\Sigma^m$ be an $m$-dimensional submanifold of a complete Riemannian manifold $M^n$ and let $r(x)=d(x,\Sigma)$ be the distance function to $\Sigma$.  

Let $\gamma:[0,t_0)\to M$ be any geodesic segment satisfying $r(\gamma(t))=t$.  If
$$Ric_k(\dot\gamma, \cdot )\geq k H$$
for some constant $H$ then for any orthonormal $k$-frame $\{e_1(t),...,e_k(t)\}\subset T_{\gamma(t)}M$ which is parallel along $\gamma$ we have
$$\sum_{i=1}^k \nabla^2 r(e_i,e_i) \leq 
\begin{cases}
k\frac{h_0\cs_H(r)-H\sn_H(r)}{\cs_H(r) + h_0\sn_H(r)} & \textup{ if } \{e_1(0),...,e_k(0)\}\subset T\Sigma \\
k\cs_H(r)/\sn_H(r) & \textup{ otherwise }
\end{cases}
$$
where $h_0=\frac{1}{k}\sum_{i=1}^k \langle S_{\dot\gamma(0)}(e_i),e_i\rangle $ and $S_{\dot\gamma(0)}$ is the Weingarten map of $\Sigma$ for the normal $\dot\gamma(0)$.
\end{theorem}

\begin{remark}
Notice that taking $\Sigma$ to be a point, the usual Hessian and Laplacian comparisons follow from this theorem by taking $k=1$ and $k=n-1$, respectively.
When $\Sigma$ is a point, the result was proved by Shen \cite[Lemma 11]{S93} and Li-Wang \cite[Theorem 1.2]{LW05}.
\end{remark}

\begin{remark}
This result implies the mean curvature comparison of \cite{GM03} when $\Sigma$ is totally geodesic.
\end{remark}

In Section \ref{sec:hessiancomparison} we give a slightly more general version of this theorem which also treats the question of rigidity when equality holds.  This comparison theorem should be compared with that of Guijarro-Wilhelm which gives comparison along a family of $k$-dimensional subspaces determined by Jacobi fields rather than parallel subspaces \cite[Lemma 2.23]{GW18}.  That comparison is based on Wilking's transverse Jacobi equation; by contrast, Theorem \ref{thm:hessiancomparison} above is based on the comparison theory for a Riccati differential equation and thus yields an elementary proof of the volume comparison of Section \ref{sec:pointwiseestimate}.

In Section \ref{sec:pointwiseestimate} we also use this Hessian comparison to generalize another Heintze-Karcher type inequality of G. Qiu and C. Xia relating the volume of a compact manifold with boundary to the total inverse mean curvature of the boundary \cite[Theorem 1.3]{QX14}.  They originally proved this inequality assuming a lower bound on sectional curvature and ask whether the inequality holds assuming only a Ricci lower bound.  In Theorem \ref{thm:hkinequality} we show that an $(n-2)$-Ricci lower bound suffices.

The paper is structured as follows.  In Section \ref{sec:prelim} we start with a brief review of the basic notions underlying the geometry of volume comparison and fix notation.  In Sections \ref{sec:hessiancomparison} and \ref{sec:pointwiseestimate} we give a proof of Theorem \ref{thm:hessiancomparison} and apply it to obtain several Heintze-Karcher type volume inequalities for $k$-Ricci curvature mentioned above.  We conclude in Section 5 with the proof of Theorem \ref{thm:integralestimate}.

\subsection*{Acknowledgements}  I would like to express gratitude to my advisor, Guofang Wei, for many helpful discussions and invaluable feedback in the preparation of this paper.  I would also like to thank Frederick Wilhelm for several suggestions to help clarify the exposition.  The author was supported in part by the National Science Foundation under the grant DMS-1506393.

\section{Preliminaries}\label{sec:prelim}

\subsection{Polar volume density, mean curvature, and distance}\label{sec:setup1}  Fix, once and for all, a complete, connected Riemannian manifold $(M^n,g)$ of dimension $n$.  Given a closed $m$-dimensional (embedded) submanifold $\Sigma$ of $M$ let $\nu=\nu(\Sigma)$ denote the normal bundle of $\Sigma$, $\nu_x$ the fiber over $x\in \Sigma$, and let $\hat\nu$ denote the unit normal bundle.  The normal bundle has a canonical Riemannian metric such that the projection is a Riemannian submersion and each tangent space splits orthogonally into a vertical subspace tangent to the fiber and a horizontal subspace consisting of vectors tangent to curves in $\nu$ which are parallel along their base curve (with respect to the normal connection).

It is well-known that the exponential map on the normal bundle $\exp_\nu:\nu\to M$ restricts to a diffeomorphism from the open neighborhood 
\begin{equation*}
U=\{u\in \nu : d(\Sigma, \exp_\nu((1+\epsilon)u)) = |(1+\epsilon)u| \textup{ for some } \epsilon > 0\}
\end{equation*}
of the zero section onto a set $\exp_\nu(U)$ such that $M\setminus\exp_{\nu}(U)$ has measure zero (see e.g. \cite{MM03,G04}).  Define the \emph{polar volume density} function $\mathcal A:\nu\to \mathbb R$ as the density of the volume element of $M$ written in $\Sigma$-polar coordinates on $U$, i.e. for normal vectors $u\in U$, $\mathcal A(u)$ can be defined using the Jacobian determinant of the normal exponential map by
$$\mathcal A(u) = |d(\exp_\nu)_{u}|\cdot|u|^{n-m-1}$$
with $\mathcal A$ extended to all of $\nu$ by setting $\mathcal A\equiv 0$ on $\nu\setminus U$.  Here, we have anticipated the use of Fubini's theorem to integrate over the normal bundle as an iterated integral whereupon the volume of the tube $T(\Sigma,r)$ can be written 
\begin{equation}\label{eq:volumeintegral}
\vol(T(\Sigma,r))=\int_0^r\int_{\hat\nu} \mathcal A(t,\xi) d\xi dt.
\end{equation}
Above and henceforth we put $\mathcal A(t,\xi)=\mathcal A(t\xi)$ for $\xi\in\hat\nu$.  Note that since $\exp_\nu$ is an isometry on the zero section of the normal bundle $\mathcal A(t,\xi)\sim t^{n-m-1}$ as $t\to 0$.

Let $r:M\to\mathbb R$ be the distance function from the submanifold $\Sigma$ and let $\Sigma_t=r^{-1}(t)\cap \exp_\nu(U)$ denote the part of the level set consisting of regular points of $r$.  For $u\in U$ with $|u|=t >0$, the gradient $\nabla r$ near $x=\exp_\nu(u)$ is a unit normal along $\Sigma_t$ and hence the Hessian $\nabla^2 r$ at $x$ is equivalent to the shape operator of the level set $\Sigma_t$ at $x$ denoted
$$S(t,\xi):T_x\Sigma_t\to T_x\Sigma_t.$$
The relative rate of change of the polar volume density in the radial direction is precisely the mean curvature $h(t,\xi)=\tr(S(t,\xi))$ of the distance level sets $\Sigma_t$; i.e.
\begin{equation}\label{eq:polardensity}
\mathcal A' = h\mathcal A.
\end{equation}

Finally, we fix  notation concerning the extrinsic geometry of $\Sigma$.  For $\xi\in\hat\nu$ based at $x\in\Sigma$, let $S_\xi:T_x\Sigma\to T_x\Sigma$ denote the Weingarten map $(\nabla\xi)^\top$ where $\xi$ is extended arbitrarily to a section of $\hat\nu$ and $^\top$ denotes projection onto $T_x\Sigma$.  Define the (normalized) mean curvature normal $\eta$ along $\Sigma$ with sign convention chosen so that for all $\xi\in\hat\nu$
$$\langle \eta,\xi\rangle = \frac{\tr(S_\xi)}{m}.$$

\subsection{Evolution equation for the shape operator}\label{sec:setup2}  
For this section we fix $\xi\in\hat\nu$ based at $x\in \Sigma$.  The one-parameter family of shape operators $S(t)=S(t,\xi)$ satisfies the Riccati differential equation
\begin{equation}\label{eq:riccati}
S' + S^2 = - R_{\partial_r}
\end{equation}
where $\partial_r=\nabla r$ denotes the gradient, $R_{\partial_r}$ denotes the directional curvature operator $R_{\partial_r}(X)=R(X,\partial_r)\partial_r$, and the prime notation denotes the covariant derivative in the direction $\partial_r$ (see e.g. \cite[Corollary 3.2.10]{P16}).

Of crucial importance in the following analysis are the initial conditions for the matrix Riccati equation above, which depends on the nature of the extension of the family of operators $S(t)$ to $t=0$.  In fact, since the square of the distance function is smooth in a neighborhood of $\Sigma$ the scaled family $tS(t)$ extends smoothly through $t=0$.  Putting $\gamma(t)=\exp_\nu(t\xi)$, if we identify the vector spaces $\dot\gamma^\perp\subset T_{\gamma(t)}M$ via parallel transport along $\gamma$ with the single vector space $E=\dot\gamma(0)^\perp$ then expanding in a Taylor series one obtains
\begin{equation}\label{eq:taylor}
S(t) = \frac{1}{t} P_{\xi} + S_\xi + O(t)
\end{equation}
where $P_\xi:E\to E$ denotes the orthogonal projection onto $\nu_x\cap \xi^\perp$ and $S_{\xi}$ is the Weingarten map extended trivially to the orthogonal complement of $T_x\Sigma$ in $E$.

\section{Hessian Comparison}\label{sec:hessiancomparison}

Using the notation from Section \ref{sec:prelim}, we now state and prove a more complete version of the Hessian comparison theorem given in the introduction.  The idea is that we can control certain partial traces of the Hessian of a distance function using the $k$-Ricci curvature.  For a linear operator $T$ on a real inner product space $\mathcal V$, and a $k$-dimensional subspace $\mathcal W\subset\mathcal V$, the partial trace of $T$ on $\mathcal W$ is defined by
$$\tr_\mathcal W(T)=\sum_{i=1}^k \langle T(e_i),e_i\rangle$$
where $\{e_1,...,e_{k}\}$ is any orthonormal basis of $\mathcal W$.  

Define the generalized trigonometric functions $\sn_H$ and $\cs_H$ by
$$\sn_H(r) = \left\{
\begin{array}{l r}
\frac{1}{\sqrt{H}}\sin(\sqrt{H}r) & H > 0 \\
r & H = 0 \\
\frac{1}{\sqrt{-H}}\sinh(\sqrt{-H} r) & H < 0
\end{array}\right.
$$
and $\cs_H(r)=\sn_H'(r)$.

\begin{lemma}[Hessian Comparison]\label{lem:hessiancomparison}
Let $\Sigma^m$ be an $m$-dimensional submanifold of a complete Riemannian manifold $M^n$.  Fix any $\xi\in\hat\nu(\Sigma)$ and put $\gamma(t)=\exp_\nu(t\xi)$.  Let $\mathcal W_0$ be any $k$-dimensional subspace of $\xi^\perp$ and let $\mathcal W_t\subset \dot\gamma(t)^\perp$ denote its parallel translation along $\gamma$.  If
$$Ric_k(\dot\gamma, \mathcal W_t)\geq kH$$
for some constant $H$, then for $t$ less than the focal distance in the direction $\xi$, we have
\begin{equation}\label{eq:model}
\tr_{\mathcal W_t}(S(t,\xi))\leq 
\left\{
\begin{array}{l r}
k\log(\cs_H(t) + w_0\sn_H(t))' & \textup{ if } \mathcal W_0\subset T\Sigma \\
k\log(\sn_H(t))' & \textup{ otherwise }
\end{array}\right.
\end{equation}
where $w_0=\tr_{\mathcal W_0}(S_\xi)/k$.  If equality holds at $t_0$, then equality holds on $(0,t_0]$ and
\begin{enumerate}
\item $H$ is an eigenvalue of $R_{\dot\gamma}$ with $\mathcal W_{t_0}$ contained in the corresponding eigenspace,
\item $\mathcal W_{t_0}^\perp$ is an invariant subspace of $R_{\dot\gamma}$, and
\item either $\mathcal W_0\subset \nu$ or $\mathcal W_0$ is contained in an eigenspace of $S_\xi$.
\end{enumerate}
\end{lemma}

\begin{proof}
Using the fact that $\mathcal W_t$ is parallel along $\gamma$, we may choose a parallel orthonormal basis $\{e_1,...,e_{n-1}\}$ for $\dot\gamma^\perp$ such that $\{e_1,...,e_k\}$ form a parallel orthonormal basis of $\mathcal W_t$.  For $0<t<t_f$, where $t_f$ is the focal distance along $\gamma$, the shape operators $S(t)=S(t,\xi)$ satisfy the Riccati equation \eqref{eq:riccati}, and since $\mathcal W_t$ is parallel along $\gamma$ the partial trace $\tr_{\mathcal W_t}$ commutes with the covariant derivative along $\gamma$ and hence 
$$\tr_{\mathcal W_t}(S)' + \tr_{\mathcal W_t}(S^2) = -Ric_k(\dot\gamma, \mathcal W_t).$$
Putting $s_{ij} = \langle Se_i,e_j\rangle$ and using the symmetry $s_{ij}=s_{ji}$ of the Hessian we have
$$\tr_{\mathcal W_t}(S^2) = \sum_{i=1}^k\sum_{j=1}^{n-1} s_{ij}^2\geq \sum_{i,j=1}^k s_{ij}^2\geq \sum_{i=1}^k s_{ii}^2\geq \frac{1}{k} \left(\sum_{i=1}^k s_{ii}\right)^2=\frac{1}{k}\tr_{\mathcal W_t} (S)^2$$
where the last inequality follows from the Cauchy-Schwarz inequality.  Putting $w(t)=\tr_{\mathcal W_t}(S(t))/k$ we have
\begin{equation}\label{eq:kriccati}
w'(t) + w(t)^2 \leq -Ric_k(\dot\gamma,\mathcal W_t)/k \leq -H.
\end{equation}
Noting that the model functions on the right hand of equation \eqref{eq:model} satisfy the Riccati equation $(f/k)' + (f/k)^2 = -H$, we may apply the comparison theory for this equation provided we match the initial conditions.  Using the Taylor expansion \eqref{eq:taylor} we see that if $\mathcal W_0\subset T_{\gamma(0)}\Sigma$ then $w(t)\to w_0$ as $t\to 0$, and otherwise $w(t)=O(t^{-1})$ as $t\to 0$.  The inequality now follows from the comparison theory for the scalar Riccati equation (see e.g. \cite[Proposition 6.4.1]{P16}).

If equality holds for some $t_0\leq t_f$, then the aforementioned Riccati comparison principle implies that equality holds on $(0,t_0]$.  From the inequalities above, it follows that with respect to the parallel basis $\{e_i\}$ the matrix representation of $S(t)$ on $(0,t_0]$ is block diagonal of the form
$$S(t)=\begin{pmatrix}
w(t)I_k & 0 \\
0 & *
\end{pmatrix}
$$
where $I_k$ is the $k\times k$ identity matrix.  Since this decomposition holds on $(0,t_0]$, it follows from the expansion \eqref{eq:taylor} that either $\mathcal W_0\subset T\Sigma$ or $\mathcal W_0\subset \nu$.  Moreover, if $\mathcal W_0\subset T\Sigma$ then \eqref{eq:taylor} implies that $\mathcal W_0$ is contained in an eigenspace of $S_\xi$ with eigenvalue $w_0=\lim_{t\to 0} w(t)$.  Parts (1) and (2) then follow from the observation that $-S'-S^2=R_{\dot\gamma}$ is also block diagonal of the same form, replacing $w(t)$ with $-w'-w^2=H$.
\end{proof}

From the proof above it is easily seen that equality is realized if $M$ is a space of constant curvature $H$ and $\Sigma$ is a submanifold such that condition $(3)$ in the lemma holds for $\Sigma$ and $\mathcal W_t$.

Our main applications of this lemma will be to the volume inequalities of the next section; however, we also recover an upper bound on the focal radius which was recently obtained by Guijarro and Wilhelm using a different Jacobi field comparison for $k$-Ricci curvature \cite{GW18}.

\begin{corollary}\label{cor:focalradius}
Let $\Sigma$ be a submanifold of a complete Riemannian manifold $M^n$ with $\dim(\Sigma)\geq k$.  If $Ric_k\geq k\cdot H>0$ then the focal radius of $\Sigma$ is at most $\frac{\pi}{2\sqrt{H}}$ and this focal radius is achieved if and only if $\Sigma$ is totally geodesic.
\end{corollary}

\begin{proof}[Proof of Corollary \ref{cor:focalradius}]
Given $x\in\Sigma$ let $\xi\in\hat\nu(\Sigma)$ be any unit normal based at $x$, and put $\gamma(t)=\exp_\nu(t\xi)$.  By replacing $\xi$ with $-\xi$ if necessary, we may assume $\langle \eta,\xi\rangle \leq 0$ where $\eta$ is the mean curvature vector of $\Sigma$.  Putting $\mathcal W_0=T_x\Sigma$ and applying the lemma, we find that $\tr_{\mathcal W_t}(S(t,\xi))$ diverges to $-\infty$ for some $t\leq\pi/(2\sqrt{H})$.  Moreover, if equality holds for all $\xi\in\hat\nu(\Sigma)$ then $\eta\equiv 0$ and it follows from part (3) of the lemma that $\Sigma$ is totally umbilic, and hence totally geodesic.
\end{proof}

\section{Volume inequalities using $k$-Ricci curvature bounds}\label{sec:pointwiseestimate}

Using the Hessian comparison above we prove two different Heintze-Karcher type volume inequalities using $k$-Ricci lower bounds.
\begin{theorem}\label{thm:pointwiseestimate}
Let $M^n$ be a complete Riemannian manifold and let $\Sigma^m$ be a closed $m$-dimensional submanifold.  Put $k=\min\{m,n-m-1\}$.  If $Ric_k\geq k\cdot H$ then
$$\vol(T(\Sigma,r))\leq \int_{\hat\nu}\int_0^{z(r,\xi)} (\cs_H(t) + \langle \eta,\xi \rangle \sn_H(t))^m \sn_H(t)^{n-m-1}dtd\xi .$$
where $z(r,\xi)$ denotes the minimum of $r$ and the first zero of the integrand.
\end{theorem}

\begin{proof}
Given $\xi\in\hat\nu(\Sigma)$ based at $x\in\Sigma$ let $\gamma(t)=\exp_\nu(t\xi)$ and let $\mathcal H_t$ and $\mathcal V_t$ denote the subspaces of $\dot\gamma(t)^\perp$ parallel along $\gamma$ to $T_x\Sigma$ and $\nu_x(\Sigma)$, respectively.  Since $\mathcal H_t$ and $\mathcal V_t$ are orthogonal, the mean curvature $h(t,\xi)$ of $\Sigma_t$ at $\gamma(t)$ is given by
$$h = \phi + \psi$$
where $\phi = \tr_{\mathcal H_t}(S(t,\xi))$ and $\psi = \tr_{\mathcal V_t}(S(t,\xi))$.  Using the assumption $Ric_k\geq kH$, Lemma \ref{lem:hessiancomparison} gives
$$h(t,\xi) \leq m \log(\cs_H(t) + \langle \eta,\xi\rangle \sn_H(t))' + (n-m-1)\log(\sn_H(t))'$$
for $0<t<t_c(\xi)$ where $t_c(\xi)$ is the distance to the cut locus of $\Sigma$ in the direction $\xi$.  From equation \eqref{eq:polardensity} we have for $0<t<t_c$ the identity $$\log(\mathcal A)' = \log[(\cs_H + \langle \eta,\xi\rangle \sn_H)^m\sn_H^{n-m-1}]'.$$
Now, since $\mathcal A(t,\xi)\sim t^{n-m-1}$ as $t\to 0$ we can integrate from $0$ to $t\leq t_c$ to obtain
$$\mathcal A(t,\xi)\leq (\cs_H + \langle \eta,\xi\rangle \sn_H)^m\sn_H^{n-m-1}.$$
The result follows from equation $\eqref{eq:volumeintegral}$.
\end{proof}

Our second application of the Hessian comparison is a generalization of another Heintze-Karcher type inequality given by G. Qiu and C. Xia in \cite{QX14}.  The inequality of Qiu-Xia is motivated by similar inequalities of A. Ros and S. Brendle which have been used to prove Alexandrov's Theorem in various contexts (see \cite{R87,B13, QX14}).  

\begin{theorem}\label{thm:hkinequality}
Let $(M^n,g)$ be a compact Riemannian manifold with smooth boundary $\Sigma$ with outward unit normal $\xi$ and mean curvature vector $\eta$.  Fix a point $p\in M$ and put $f(x)=\cosh(r(x))$ where $r(x)=d(x,p)$ is the distance from $x$ to $p$ in $M$.  

If $\langle \eta,\xi\rangle>0$ everywhere on $\Sigma$ and $Ric_{n-2}\geq -(n-2)$ then 
\begin{equation}
\int_\Sigma\frac{f}{\langle \eta,\xi\rangle}dvol_\Sigma \geq \int_M (\Delta f)dvol_M.
\end{equation}
Equality holds if and only if $M$ is a geodesic ball in a space form of constant sectional curvature $-1$.
\end{theorem}

\begin{remark}
This theorem was proved in \cite{QX14} using the assumption $Ric_1\geq -1$ (i.e. $\sec\geq -1$) in place of $Ric_{n-2}\geq -(n-2)$.  The theorem above shows that an $(n-2)$-Ricci lower bound suffices.
\end{remark}

The proof is based on the following weighted version of Reilly's formula.

\begin{theorem}[Qiu-Xia, 2014]\label{thm:reilly}
Let $(M^n,g)$ be a compact Riemannian manifold with smooth boundary $\Sigma$ and outward unit normal $\xi$ and let $f:M\to\mathbb R$ be a.e. twice differentiable.  Given $u\in C^\infty(M)$ and $\lambda \in \mathbb R$ such that $u\vert_\Sigma=u_0$ is constant
\begin{align*}
 (n-1)\int_\Sigma& f \langle \eta,\xi\rangle \langle \nabla u, \xi\rangle ^2 + 2\lambda f \langle \nabla u, \xi\rangle u_0-  \langle \nabla f , \xi\rangle \lambda u_0^2  dvol_\Sigma \\
&=\int_M f\left( (\Delta u + n\lambda u)^2 - |\nabla^2 u + \lambda u g|^2  \right) -  (n-1) \lambda (\Delta f + n \lambda f)u^2 dvol_M \\
&+\int_M \left( \Delta f g - \nabla^2 f - fRic + 2(n-1)\lambda f g \right)(\nabla u,\nabla u) dvol_M
\end{align*}
where $\eta$ is the mean curvature vector of $\Sigma$.
\end{theorem}

\begin{proof}[Proof of Theorem \ref{thm:hkinequality}]
First observe that
$$\nabla^2 f =\cosh(r) dr^2 +  \sinh(r) \nabla^2 r.$$
For any $x$ outside the cut locus of $p$ the Hessian $\nabla^2 f$ thus has an orthonormal frame of eigenvectors $\{e_1,...,e_{n-1},\nabla r\}$ at $x$ with dual frame $\{\theta^1,...,\theta^{n-1},dr\}$ so that
$$\nabla^2 f  = \cosh(r)dr^2 + \sinh(r)\sum_{i=1}^{n-1} \kappa_i(\theta^i)^2$$
where $\kappa_i=\nabla^2 r(e_i,e_i)$.  It follows that
$$(\Delta f) g - \nabla^2 f  = \sinh(r)(\Delta r)dr^2 + \sum_{i=1}^{n-1}\Big(\cosh(r) + \sinh(r)\sum_{j\neq i} \kappa_j\Big)(\theta^i)^2.$$
Using the assumption $Ric_{n-2}\geq -(n-2)$, the Hessian comparison theorem above then implies that $\Delta r \leq (n-1) \cosh(r)/\sinh(r)$ and $\sum_{j\neq i} \kappa_j\leq (n-2)\cosh(r)/\sinh(r)$ and hence
\begin{equation}\label{eq:substaticpotential}
(\Delta f) g - \nabla^2 f \leq (n-1)fg.
\end{equation}
The rest of the proof is the same as that in \cite{QX14}, but we include it for completeness.  Put $\lambda = -1$ and let $u$ be the solution to the Dirichlet boundary value problem
\begin{align*}
\left\{
\begin{array}{l r}
\Delta u = nu  \\
u\vert_\Sigma = c>0 
\end{array}\right.
\end{align*}
From equation \eqref{eq:substaticpotential} it follows that
$$\left(\Delta f g - \nabla^2 f - fRic - 2(n-1) f g\right)(\nabla u,\nabla u) \leq 0$$
and
$$(n-1) (\Delta f - n  f)u^2\leq 0$$
and hence the formula of Theorem \ref{thm:reilly} gives
\begin{equation}\label{eq:reilly}
\int_\Sigma f \langle \eta,\xi\rangle \langle \nabla u, \xi\rangle ^2  \leq \int_\Sigma \left( 2f \langle \nabla u, \xi\rangle c-  \langle \nabla f , \xi\rangle c^2  \right) .
\end{equation}
Now, using H\"older's inequality followed by the previous inequality we get
\begin{align*}
\left(  \int_\Sigma  f \langle \nabla u,\xi\rangle \right)^2 &\leq \int_\Sigma \frac{f}{\langle \eta,\xi\rangle}\int_\Sigma f\langle \eta,\xi\rangle\langle \nabla u, \xi\rangle^2  \\
&\leq \int_\Sigma \frac{f}{\langle \eta,\xi\rangle}\int_\Sigma \left( 2f \langle \nabla u, \xi\rangle c-  \langle \nabla f , \xi\rangle c^2  \right) .
\end{align*}
It follows that
$$\left(\int_\Sigma f\langle \nabla u , \xi\rangle  - c\int_\Sigma \frac{f}{\langle \eta,\xi\rangle} \right)^2 - c^2\left(\int_\Sigma \frac{f}{\langle \eta,\xi\rangle} \right)^2 \leq -c^2 \int_\Sigma \frac{f}{\langle \eta,\xi\rangle} \int_\Sigma \langle \nabla f,\xi\rangle. $$
Dropping the first term (since it's nonnegative) and noting that $c\neq 0$ we obtain
$$\int_\Sigma\frac{f}{\langle \eta,\xi\rangle} \geq \int_\Sigma \langle \nabla f , \xi\rangle =\int_M \Delta f$$
as desired.  If equality holds than equality must hold in equation \ref{eq:reilly} and hence $(\Delta u - n u)^2 = |\nabla^2 u - u g|^2$.  By construction $\Delta u = nu$ and hence in the equality case we have $\nabla^2 u = ug$ in $M$.  Since $u\vert_\Sigma=c$ it follows from an Obata type rigidity result that $M$ must be a geodesic ball in a space form of constant curvature $-1$ (see e.g. \cite[Theorem 5.1]{WY14}).
\end{proof}

\section{Volume estimates using integral curvature bounds}\label{sec:integralestimate}

We now give the proof of Theorem \ref{thm:integralestimate}.  In particular, throughout this section we assume $0<m<n-1$.  In the proof of Theorem \ref{thm:pointwiseestimate} we assumed a pointwise lower curvature bound and used the partial traces of the Riccati equation \eqref{eq:riccati} to bound the mean curvature $h$ explicitly and thus bound the logarithmic growth of the volume density $\mathcal A$.  Unfortunately, in order to obtain a bound depending only on integrals of the curvature this approach fails since one cannot use the comparison theory for the Riccati differential equation.

As before, we fix $\xi\in \hat \nu$ and put $\gamma(t)=\exp_\nu(t\xi)$.  Differentiating eq. \eqref{eq:polardensity} gives
\begin{equation}\label{eq:densityeq}
\mathcal A''=(h'+h^2)\mathcal A.
\end{equation}
Taking the trace of equation \eqref{eq:riccati} gives $h'+\tr(S^2)=-Ric(\dot\gamma,\dot\gamma)$ which leaves us to control the second order invariant $\tr(S)^2 - \tr(S^2)$ in \eqref{eq:densityeq} in terms of curvature.  This motivates us \cite{G88,Y92} to consider in place of $\mathcal A$ the function $A$ with $\mathcal A=A^{n-1}$ which satisfies
$$A''=\frac{1}{n-1}\left(h'+\frac{h^2}{n-1}\right)A\leq -\frac{Ric(\dot\gamma,\dot\gamma)}{n-1}A$$
allowing us to control the second order invariant of $S$ using the Cauchy-Schwarz inequality.  However, if $0<m<n-1$ then the initial conditions for the function $A$ are unusable, namely $A(0)=0$ and $A'(0)=\infty$.  

The remarkable observation of \cite{PSW97} was that one can control certain products of eigenvalues of $S$ directly in terms of integrals of sectional curvature without relying on the Cauchy-Schwarz inequality, provided one of the eigenvalues vanishes at $\Sigma$.  For $m=1$, putting $\mathcal A=A^{n-2}$ (so that $A'(0)=1$) and assuming $\Sigma$ is a geodesic then allows control of the second order invariant of $S$.  However, generalizing this directly to higher dimensional submanifolds by setting $\mathcal A=A^{n-m-1}$ then yields an estimate only when $\Sigma$ is totally geodesic.

Instead, motivated by the pointwise comparison of the previous section, we decompose the mean curvature as $h=\phi + \psi$ as in the proof of Theorem \ref{thm:pointwiseestimate}, and then decompose the polar volume density $\mathcal A$ into two functions $\mathcal A(t)=\mathcal J(t)\mathcal Y(t)$ where $\mathcal J(t)$ is defined by the equation
\begin{equation*}
\left\{
\begin{array}{l}	
\mathcal J' = \phi \mathcal J, \\
\mathcal J(0)=1
\end{array}\right.
\end{equation*}
from which it follows that $\mathcal Y$ satisfies $\mathcal Y' = \psi \mathcal Y$.
Putting $\mathcal J= J^m$ and $\mathcal Y = Y^{n-m-1}$ we have $\mathcal A = J^m Y^{n-m-1}$ with
\begin{equation}\label{eq:Jineq}
J'' = \frac{1}{m}\left(\phi' + \frac{\phi^2}{m}\right) J \leq -\frac{Ric_m(\dot\gamma,\mathcal H_t)}{m}J
\end{equation}
and
\begin{equation}\label{eq:Yineq}
Y'' = \frac{1}{n-m-1}\left(\psi' + \frac{\psi^2}{n-m-1}\right) Y \leq -\frac{Ric_{n-m-1}(\dot\gamma,\mathcal V_t)}{n-m-1}Y.
\end{equation}
The initial conditions for $J$ and $Y$ are easily found to be
$$\left\{
\begin{array}{l}
J(0)=1, \\ 
J'(0) = \langle\eta,\xi\rangle,
\end{array}\right. \;\;\;\;\;
\left\{\begin{array}{l} Y(0)=0, \\ Y'(0)=1.
\end{array}\right.
$$

The main challenge introduced with this decomposition is that we only want to consider expressions involving curvature multiplied by the full volume density $\mathcal A$, rather than curvature multiplied by just the function $J$ or $Y$ as in \eqref{eq:Jineq} and \eqref{eq:Yineq}.  With this in mind, rather than integrating the two inequalities directly these considerations motivate the following lemma.
\begin{lemma}\label{lem:lemma1}
If $0<m<n-1$ the functions $J$ and $Y$ defined above satisfy
\small
\begin{equation}\label{eq:lemma1}
J'(t)Y^{\frac{n-m-1}{m}}(t) \leq \int_0^t(\rho_m)_\m \mathcal A^{1/m}ds + \frac{1}{m^2}\int_0^t(\phi_\p\psi_\p)\mathcal A^{1/m}ds	
\end{equation}
and
\small
\begin{equation*}\label{eq:lemma2}
Y'(t)J^{\frac{m}{n-m-1}}(t)\leq 1 + \int_0^t (\rho_{n-m-1})_\m \mathcal A^{\frac{1}{n-m-1}}ds + \frac{1}{(n-m-1)^2}\int_0^t (\phi_\p\psi_\p)\mathcal A^{\frac{1}{n-m-1}}ds.
\end{equation*}
\normalsize
\end{lemma}

\begin{proof}
For the first inequality, for any $\delta>0$ we have
$$(J'Y^\delta)' = J''Y^\delta + \delta Y^{\delta - 1} J'Y'.$$
Using the identities $J'=(\phi/m)J$ and $Y'= [\psi/(n-m-1)]Y$ together with eqs. \eqref{eq:Jineq} and \eqref{eq:Yineq} we get
$$(J'Y^\delta)'\leq \left(-\rho_m + \delta \frac{\phi\psi}{m(n-m-1)}\right)JY^\delta.$$
Now, if $J'(t)\leq 0$ the inequality \eqref{eq:lemma1} holds automatically so we only need to show the inequality for all values of $t$ such that $J'(t)> 0$.  Moreover, since $J'(0)Y^\delta(0)=0$ it follows that all such values of $t$ are contained in an interval $[t_0,t]$ such that $J'(t_0)Y^\delta(t_0)=0$ and $J'\geq 0$ on $[t_0,t]$.  On such an interval, $J'\geq 0$ implies $\phi\geq 0$ and hence $\phi\psi\leq\phi_\p\psi_\p$ on $[t_0,t]$.  Using also that $-\rho_m\leq (\rho_m)_\m$ and integrating over the interval $[t_0,t]$ gives
$$J'(t)Y^\delta(t)\leq \int_{t_0}^t (\rho_m)_\m JY^\delta + \delta\int_{t_0}^t \frac{\phi_\p\psi_\p}{m(n-m-1)}JY^\delta.$$  Since the integrands are nonnegative, we can replace the lower bound $t_0$ with $0$ and preserve the inequality.  Finally, taking $\delta=(n-m-1)/m$ gives the result \eqref{eq:lemma1}.

Analogous reasoning leads to the second inequality, except that the initial condition $Y'(0)J^\delta(0)=1  $ leads to the extra term on the right hand side of the inequality.
\end{proof}

Based on this lemma, one now only needs to control the product $\phi_\p\psi_\p$ in terms of curvature.  We prove that this is possible provided $\phi_\p$ vanishes at $\Sigma$, generalizing the eigenvalue estimate in \cite{PSW97}.

\begin{lemma}\label{lem:2meanlemma}
Let $\phi,\psi,$ and $\mathcal A$ be as above.  Put $k=\min\{m,n-m-1\}$.  If $\phi_\p(t)\psi_\p(t)$ is bounded as $t\to 0$ then for any $p>n-k$,
$$\left(\int_0^t(\phi_\p\psi_\p)^p\mathcal A ds\right)^\frac{1}{p}\leq \frac{2p-1}{p-(n-k)}\left(\int_0^t (\rho_{k})_\m^p\mathcal A ds\right)^\frac{1}{p}.$$
\end{lemma}

The proof is given at the end of this section and is a straightforward modification of the proof in \cite{PSW97}.  Note that if the submanifold $\Sigma$ is minimal then it follows from equation \eqref{eq:taylor} that $\phi_\p\psi_\p$ is bounded as $t\to 0$.

Using these inequalities, we now consider the area of the equidistant hypersurfaces $v(t)=\vol(\Sigma_t)$ given by the integral
\begin{equation}\label{eq:areaintegral}
v(t)=\int_{\hat\nu}J^m(t,\xi)Y^{n-m-1}(t,\xi) d\xi.
\end{equation}
The volume $V(r)=\vol(T(\Sigma,r))$ can then be written
$$V(r)=\int_0^r v(t)dt.$$     We wish to differentiate $v(t)$ via the expression \eqref{eq:areaintegral}.  Note that the integrand is smooth and nonnegative on the open set $U$ defined in Section \ref{sec:prelim} and vanishes on $\nu\setminus U$, but may be discontinuous on the boundary of $U$.  However, since $U$ is star-shaped with respect to the zero section of $\nu$ (i.e. $u\in U$ implies $\lambda u \in U$ for $0\leq \lambda\leq 1$) it follows that $v$ is an almost everywhere differentiable lower semi-continuous function (see \cite{A07}) and
\begin{align*}
v'(t)&\leq \int_{\hat\nu} mJ^{m-1}Y^{n-m-1}J' + (n-m-1)J^mY^{n-m-2}Y'd\xi.
\end{align*}
We now substitute the two inequalities of Lemma \ref{lem:lemma1} and use two applications of H\"older's inequality.  For example, the first term satisfies
\small
\begin{align*}
m&\int_{\hat\nu} J^{m-1}Y^{n-m-1}J'd\xi \leq m\int_{\hat\nu} \mathcal A^\frac{m-1}{m}\left(\int_0^t(\rho_m)_\m\mathcal A^\frac{1}{m}ds + \frac{1}{m^2}\int_0^t (\phi_\p\psi_\p)\mathcal A^\frac{1}{m}ds\right)d\xi \\
&\leq mv(t)^\frac{m-1}{m}\left[\left(\int_{\hat\nu} \Big(\int_0^t (\rho_m)_\m\mathcal A^\frac{1}{m}ds\Big)^m d\xi\right)^\frac{1}{m} + \frac{1}{m^2}\left(\int_{\hat\nu}\Big(\int_0^t (\phi_\p\psi_\p)\mathcal A^\frac{1}{m}ds\Big)^m d\xi\right)^\frac{1}{m}\right] \\
&\leq mv(t)^\frac{m-1}{m}t^\frac{m-1}{m}\left[\left(\int_{\hat\nu}\int_0^t (\rho_m)_\m^m \mathcal A dsd\xi\right)^\frac{1}{m} + \frac{1}{m^2}\left(\int_{\hat\nu}\int_0^t (\phi_\p\psi_\p)^m\mathcal Adsd\xi\right)^\frac{1}{m}\right].
\end{align*}
\normalsize
Handling the second term of the integral in a similar fashion one easily checks that
\small
\begin{equation}
\begin{split}\label{eq:areaderivative}
v'(t&) \leq (n-m-1)\vol(\hat\nu)^\frac{1}{n-m-1}v(t)^{\frac{n-m-2}{n-m-1}} \\
&+ \left[(n-m-1)\norm{(\rho_{n-m-1})_\m}_{{n-m-1},t} + \tfrac{1}{n-m-1}\norm{\phi_\p\psi_\p}_{{n-m-1},t} \right](tv(t))^{\frac{n-m-2}{n-m-1}} \\
&+ \left[m\norm{(\rho_m)_\m}_{m,t} + \tfrac{1}{m}\norm{\phi_\p\psi_\p}_{m,t} \right](tv(t))^{\frac{m-1}{m}}
\end{split}
\end{equation}
\normalsize
where $\norm{\cdot}_{p,t}$ is the usual $L^p$ norm on the tube $T(\Sigma,t)$.  In order to make use of the estimate in Lemma \ref{lem:2meanlemma} it is necessary to raise the exponents in the expression above at the cost of a volume term via the inequality
\begin{equation}\label{eq:holder}
\norm{f}_{q,t}\leq \norm{f}_{p,t}V(t)^{\frac{1}{q}-\frac{1}{p}}
\end{equation}
provided $p\geq q\geq 1$.

Henceforth, set $k=\min\{m,n-m-1\}$ and note that $\rho_{k}\leq{\rho_{n-k-1}}$.  Using the inequality \eqref{eq:holder} together with Lemma \ref{lem:2meanlemma}, we have for any $p>n-k\geq q$
$$\norm{\phi_\p\psi_\p}_{q,t}
\leq \frac{2p-1}{p-(n-k)}V(t)^{\frac{1}{q}-\frac{1}{p}}\norm{(\rho_k)_\m}_{p,t}.$$
Applying these observations to \eqref{eq:areaderivative} we obtain
\begin{align*}
V''(t)&\leq (n-m-1)\vol(\hat\nu)^\frac{1}{n-m-1}V'(t)^\frac{n-m-2}{n-m-1} \\
&\;\;\;\;\;+ \norm{(\rho_k)_\m}_{p,t}\left(C_1 V(t)^{\frac{1}{m}-\frac{1}{p}}(tV'(t))^\frac{m-1}{m}
+ C_2V(t)^{\frac{1}{n-m-1}-\frac{1}{p}}(tV'(t))^\frac{n-m-2}{n-m-1}\right)
\end{align*}
where 
\begin{align*}
C_1&=m+\frac{2p-1}{m(p-n+k)}, \\
C_2&=(n-m-1)+\frac{2p-1}{(n-m-1)(p-n+k)}.
\end{align*}
In order to obtain an inequality which depends more generally on $\norm{(\rho_k-H)_\m}$ we observe that for $H\leq 0$,
$$(\rho_k)_\m \leq (\rho_k - H)_\m + |H|$$
and hence $(\rho_k)_\m^p  \leq 2^{p-1}\left((\rho_k - H)_\m^p + |H|^p\right).$  It then follows that
$$\norm{(\rho_k)_\m}_{p,t} \leq 2^\frac{p-1}{p}\left( \norm{(\rho_k - H)_\m}_{p,t} + |H|V(t)^\frac{1}{p} \right).$$  Substituting this back into the inequality above and estimating $2^\frac{p-1}{p}< 2$ we obtain
\small
\begin{align*}
V''(t)&\leq (n-m-1)\vol(\hat\nu)^\frac{1}{n-m-1}V'(t)^\frac{n-m-2}{n-m-1} \\
&\quad + 2\norm{(\rho_k-H)_\m}_{p}\left(C_1 V(t)^{\frac{1}{m}-\frac{1}{p}}(tV'(t))^\frac{m-1}{m}
+ C_2V(t)^{\frac{1}{n-m-1}-\frac{1}{p}}(tV'(t))^\frac{n-m-2}{n-m-1}\right) \\
&\quad + 2|H|\left(C_1 V(t)^{\frac{1}{m}}(tV'(t))^\frac{m-1}{m}
+ C_2V(t)^{\frac{1}{n-m-1}}(tV'(t))^\frac{n-m-2}{n-m-1}\right).
\end{align*}
\normalsize
It remains to use this differential inequality to obtain an estimate for $V(t)$.  To simplify notation we introduce the constants 
\begin{align*}
a &= (n-m-1)(\vol(\mathbb S^{n-m-1})\vol(\Sigma))^\frac{1}{n-m-1}, \\
b &= \norm{(\rho_k-H)_\m}_p, \\
c&=2(n-k-1) + \frac{2}{k}\left(\frac{2p-1}{p-n+k}\right).
\end{align*}
Noting that $\vol(\hat\nu) = \vol(\mathbb S^{n-m-1})\vol(\Sigma)$ the previous inequality then implies
\begin{align*}
V''(t)&\leq a (V')^{1-\frac{1}{n-m-1}} + cb\left(V^{\frac{1}{m}-\frac{1}{p}}(tV')^{1-\frac{1}{m}} + V^{\frac{1}{n-m-1}-\frac{1}{p}}(tV')^{1-\frac{1}{n-m-1}}\right) \\
&\quad+c|H|\left( V(t)^{\frac{1}{m}}(tV'(t))^{1-\frac{1}{m}} + V(t)^{\frac{1}{n-m-1}}(tV'(t))^{1-\frac{1}{n-m-1}}\right).
\end{align*}
Multiplying through by the nonnegative quantity $(V')^{\frac{1}{n-k-1}}$ and putting 
\begin{align*}
\delta_1 = \frac{1}{n-m-1}-\frac{1}{n-k-1},
\qquad\delta_2 =\frac{1}{k}-\frac{1}{n-k-1},
\qquad\delta_3 = \frac{1}{n-k-1} - \frac{1}{p},
\end{align*}
and $\alpha = (n-k-1)/(n-k)$ gives
\begin{align}
\begin{split}\label{eq:secondorderineq}
V''(t)(V')^\frac{1}{n-k-1}&\leq 
a(V')^{1-\delta_1}+cb t^\frac{k-1}{k}V^{\delta_2+\delta_3}(V')^{1-\delta_2} + \tfrac{cb}{{1+\delta_3}} t^\frac{n-k-2}{n-k-1}(V^{1+\delta_3})' \\
&\quad + c|H| t^\frac{k-1}{k}V^{\frac{1}{k}}(V')^{1-\delta_2} + \alpha c|H| t^\frac{n-k-2}{n-k-1}(V^{1/\alpha})'
\end{split}
\end{align}
We now integrate both sides from $0$ to $t$ as follows.  Noting that $0\leq \delta_1,\delta_2,\delta_3 <1$ and using H\"older's inequality we get the inequalities
\begin{align*}
\int_0^t (V')^{1-\delta_1} ds &\leq t^{\delta_1}\left(\int_0^t V' ds\right)^{1-\delta_1}=t^{\delta_1} V(t)^{1-\delta_1}
\end{align*}
and
\begin{align*}
\int_0^t s^\frac{k-1}{k} V^{\delta_2+\delta_3} (V')^{1-\delta_2} ds &\leq t^\frac{k-1}{k}\int_0^t V^{\delta_2+\delta_3}(V')^{1-\delta_2}ds \\
&\leq t^{\frac{k-1}{k}}t^{\delta_2}\left(\int_0^t V^{(\delta_2+\delta_3)/(1-\delta_2)}(V')ds\right)^{1-\delta_2} \\
&= t^\frac{n-k-2}{n-k-1}\left(\frac{1-\delta_2}{1+\delta_3}\right)^{1-\delta_2} V^{1+\delta_3}.
\end{align*}
Handling the integral of the fourth term on the right hand side of equation \eqref{eq:secondorderineq} in the same manner, we integrate equation \eqref{eq:secondorderineq} from $0$ to $t$ to obtain
\begin{align*}
(V')^{1/\alpha}&\leq \frac{1}{\alpha}\left(a t^{\delta_1}V^{1-\delta_1}+cb\left[\left(\tfrac{1-\delta_2}{{1+\delta_3}}\right)^{1-\delta_2} + \tfrac{1}{{1+\delta_3}}\right]t^\frac{n-k-2}{n-k-1}V^{1+\delta_3}\right)\\
&\quad + \frac{c|H|}{\alpha}\left(\left[\left(1-\tfrac{\alpha}{k}\right)^{1-\delta_2} + \alpha \right]t^\frac{n-k-2}{n-k-1}V^{1/\alpha} \right).
\end{align*}
Noting that both quantities in brackets are bounded above by $2$ and since $0<\alpha<1$ the inequality $(x+y)^\alpha\leq x^\alpha+y^\alpha$ for $x,y\geq 0$ implies\begin{equation*}
V'\leq \alpha^{-\alpha}\left(a t^{\delta_1}V^{1-\delta_1} + 2cb t^{\frac{n-k-2}{n-k-1}} V^{1+\delta_3}\right)^\alpha + \left(2c|H|/\alpha\right)^\alpha t^{2\alpha-1} V
\end{equation*}
Multiplying by the integrating factor $\mu(t)=e^{-\kappa t^{2\alpha}}$ where $\kappa = (2c|H|/\alpha)^\alpha/(2\alpha)$ transforms this inequality into
\begin{equation*}
(\mu V)' \leq \alpha^{-\alpha} \left(a t^{\delta_1}V^{1-\delta_1} + 2cb t^{\frac{n-k-2}{n-k-1}} V^{1+\delta_3}\right)^\alpha \mu
\end{equation*}
and using the fact that $0<\mu\leq 1$ for $t\geq 0$ we can write
\begin{equation}\label{eq:firstorderineq}
(\mu V)' \leq \alpha^{-\alpha} \left(a t^{\delta_1}(\mu V)^{1-\delta_1} + 2cb t^{\frac{n-k-2}{n-k-1}} (\mu V)^{1+\delta_3}\right)^\alpha.
\end{equation}
Put $r_0=\inf \{ r : b\mu(r)V(r)\geq 1\}$ with $r_0=\infty$ if $b\mu(r)V(r)<1$ for all $r>0$.  Define the function $f:[0,\infty)\to \mathbb R$ by
\begin{equation*}
f(t)=
\begin{cases}
\mu(t)V(t) & \textup{ if } t\leq r_0 \\
\max\{\mu(t)V(t),1/b\} & \textup{ if } t>r_0
\end{cases}
\end{equation*}
and observe that $f$ is absolutely continuous and satisfies the differential inequality \eqref{eq:firstorderineq} with $f$ in place of $\mu V$.  We now use this inequality to derive an upper bound for the function $f(t)$.  
To integrate this inequality, first notice that the exponents satisfy $1-\delta_2 < 1+\delta_3$.  On the interval $[0,r_0]$, since $bf\leq 1$ we thus have
\begin{flalign*}
&&f' \leq \alpha^{-\alpha}\left(a t^{\delta_1} + 2cb^{1-\delta_1-\delta_3} t^{\frac{n-k-2}{n-k-1}}  \right)^\alpha f^{1-\frac{\alpha}{n-m-1}}, &&(t\leq r_0).
\end{flalign*}
Putting $\beta=\delta_1+\delta_3$ it follows that for $0\leq t \leq r_0$ there holds
\begin{flalign*}
&&\left(f^\frac{\alpha}{n-m-1}\right)' \leq \tfrac{1}{n-m-1}\alpha^{1-\alpha}\left(a t^{\delta_1} + 2cb^{1-\beta} t^{\frac{n-k-2}{n-k-1}}\right)^\alpha, && (t\leq r_0).
\end{flalign*}
Integrating the right hand side from $0$ to $r\leq r_0$ using H\"older's inequality we find
\begin{equation*}
\int_0^r \left(at^{\delta_1} + 2cb^{1-\beta}t^\frac{n-k-2}{n-k-1}\right)^\alpha dt \leq r^{1-\alpha}\left(\int_0^r at^{\delta_1} + 2cb^{1-\beta}t^\frac{n-k-2}{n-k-1} dt\right)^\alpha
\end{equation*}
and so carrying out the integration yields
\begin{flalign*}
&&f(r)\leq \left(\tfrac{\alpha^{1-\alpha}}{n-m-1}\right)^\frac{n-m-1}{\alpha}\left(\frac{a}{1+\delta_1}r^\frac{n-m}{n-m-1} + \frac{2cb^{1-\beta}}{2-\frac{1}{n-k-1}} r^2 \right)^{n-m-1}, && (r\leq r_0).
\end{flalign*}
Simplifying the expression, noting that $1/(1+\delta_1)\leq 1$, $2-1/(n-k-1)\geq 1$, and $\alpha^{1-\alpha}/(n-m-1)\leq 1$ we finally obtain
\begin{flalign}\label{eq:volumeineq1}
&& f(r)\leq w(r)^{n-m-1}, && (r\leq r_0)
\end{flalign}
where
$$
w(r) = \left( \tfrac{\alpha}{n-m-1}\right)^\frac{1}{n-k-1}\vol(\hat\nu)^\frac{1}{n-m-1} r^\frac{n-m}{n-m-1} +2cb^{1-\beta} r^2.
$$
For $t\geq r_0$ we may assume $b\neq 0$ and since $bf\geq 1$ we have
\begin{flalign*}&&f'\leq \alpha^{-\alpha}\left(a b^{\delta_1+\delta_3} t^{\delta_1}+2c b t^\frac{n-k-2}{n-k-1}\right)^\alpha f^{1-\frac{\alpha}{p}}, &&(t\geq r_0).
\end{flalign*}
Proceeding as above except that we integrate from $r_0$ to $r>r_0$ it is easy to check that for $r\geq r_0$,
\begin{equation*}
f(r)\leq \frac{1}{b}\left[1+ b^\frac{\alpha}{n-m-1}w(r)^\alpha\right]^{p/\alpha}.
\end{equation*}
Moreover, for $r\geq r_0$ we have $b w(r)^{n-m-1}\geq b w(r_0)^{n-m-1}\geq bf(r_0)= 1$ and hence
$$b^{-1}(1+b^\frac{\alpha}{n-m-1}w(r)^\alpha)^{p/\alpha}\leq b^{-1}(2b^\frac{\alpha}{n-m-1}w(r)^\alpha)^{p/\alpha} = 2^{p/\alpha} b^{\beta p}w(r)^p$$
and thus for $r\geq r_0$ we have
\begin{flalign}\label{eq:volumeineq2}
&& f(r)\leq 2^{p/\alpha}b^{\beta p}w(r)^{p}, && (r\geq r_0).
\end{flalign}
Combining equations \eqref{eq:volumeineq1} and \eqref{eq:volumeineq2} it then follows that for all $r\geq 0$ the function $f$ satisfies
\begin{equation*}
f(r)\leq w(r)^{n-m-1} + 2^{p/\alpha}b^{\beta p} w(r)^p.
\end{equation*} 
Noting that $f(r)\geq \mu(r)V(r)$ it follows that the volume of the tube $T(\Sigma,r)$ satisfies
\begin{equation*}
V(r)\leq \left(w(r)^{n-m-1} + 2^{p/\alpha}b^{\beta p} w(r)^p\right)e^{\kappa r^{2\alpha}}
\end{equation*}
This completes the proof of Theorem \ref{thm:integralestimate} contingent on our proof of Lemma \ref{lem:2meanlemma}.

\begin{remark}\label{rem:pointwise}
As mentioned in the introduction, in the case of a pointwise lower curvature bound $\norm{(\rho_k)_\m}_p=0$ with $k=m$, the estimate reduces to
\begin{equation}
V(r)\leq
\frac{1}{n-m}\vol(\Sigma)\vol(\mathbb S^{n-m-1})r^{n-m}
\end{equation}
which is precisely the volume of a tube around a piece of an $m$-plane in $\mathbb R^n$.  The loss of sharpness in the pointwise case when $k\neq m$ comes from the use of H\"older's inequality above, and could be removed by setting $\norm{(\rho_k)_\m}_p=0$ earlier in the computation.
\end{remark}

\begin{proof}[Proof of Lemma \ref{lem:2meanlemma}]
Define 
\begin{align*}
\sigma &= \min\{\phi_\p/m,\;\psi_\p/(n-m-1)\} \\
\tau &= \max\{\phi_\p/m,\;\psi_\p/(n-m-1)\}
\end{align*}
and observe that both $\sigma$ and $\tau$ are absolutely continuous and from equation \eqref{eq:kriccati}, using the fact that $0\leq (\rho_{n-k-1})_\m\leq (\rho_k)_\m$ they satisfy
\begin{align*}
\sigma' + \sigma^2&\leq(\rho_k)_\m  \\
\tau' + \tau^2 &\leq (\rho_k)_\m .
\end{align*}
Multiplying the first equation by $(\sigma\tau)^{p-1}\mathcal A$ and integrating, we have
\begin{equation}\label{eq:2mean1}
\int_0^r \sigma'(\sigma\tau)^{p-1}\mathcal A dt + \int_0^r \sigma^{p+1}\tau^{p-1}\mathcal A \leq \int_0^r(\rho_k)_\m (\sigma\tau)^{p-1}\mathcal Adt.
\end{equation}
Integrating the first term by parts we find that
\begin{align*}
\frac{1}{p}&\int_0^r (\sigma^p)' \tau^{p-1}\mathcal A dt = \frac{1}{p}\sigma^p\tau^{p-1}\mathcal A\Big\vert_0^r - \frac{p-1}{p}\int_0^r\sigma^p\tau^{p-2}\tau'\mathcal Adt -\frac{1}{p}\int_0^r\sigma^p\tau^{p-1}h\mathcal Adt \\
&\quad\geq 0 - \frac{p-1}{p}\int_0^r\sigma^p\tau^{p-2}((\rho_k)_\m - \tau^2)\mathcal Adt - \frac{n-k-1}{p}\int_0^r\sigma^p\tau^{p-1}(\sigma+\tau)\mathcal Adt
\end{align*}
where we have used the fact that $\sigma\tau$ is bounded as $t\to 0$ and $\mathcal A(0)=0$ for $m<n-1$.  The last term uses the observation $h=\phi+\psi\leq (n-k-1)(\sigma + \tau)$.  Substituting back into \eqref{eq:2mean1}, we now have
\begin{align*}
\frac{p-(n-k)}{p}\int_0^r (\sigma\tau)^p\mathcal A dt &+\left(1-\frac{n-k-1}{p}\right)\int_0^r \sigma^{p+1}\tau^{p-1}\mathcal A dt \\ 
&\;\;\leq \frac{p-1}{p}\int_0^r(\rho_k)_\m \sigma^{p}\tau^{p-2}\mathcal A dt + \int_0^r (\rho_k)_\m (\sigma\tau)^{p-1}\mathcal Adt.
\end{align*}
Assuming $p>n-k$ the first term is positive, the second term is non-negative and can be dropped, and since $0\leq \sigma\leq \tau$ we can use $\sigma^p\tau^{p-2}\leq (\sigma\tau)^{p-1}$ to obtain
$$\int_0^r (\sigma\tau)^p\mathcal Adt\leq \frac{2p-1}{p-(n-k)}\int_0^r (\rho_k)_\m (\sigma\tau)^{p-1}\mathcal Adt.$$
Finally, using H\"older's inequality on the right hand side we have
$$\int_0^r(\rho_k)_\m (\sigma\tau)^{p-1}\mathcal Adt\leq \left(\int_0^r (\rho_k)_\m ^p\mathcal A dt\right)^{1/p}\left(\int_0^r(\sigma\tau)^{p}\mathcal A dt\right)^{1-\frac{1}{p}}$$
and the lemma follows immediately.
\end{proof}

We conclude with a proof of Corollary \ref{cor:cheegerlemma}.
\begin{proof}[Proof of Corollary \ref{cor:cheegerlemma}]
Fix $n,m,p,H,v_0,D$ as in the statement of the Corollary.  By Theorem \ref{thm:integralestimate}, there exists a function $F(a,b,r)$ with the property that $F\to 0$ as $a,b\to 0$ such that for any closed $m$-dimensional minimal submanifold $\Sigma$ of a complete $n$-dimensional Riemannian manifold $M$ the volume of the tube around $\Sigma$ satisfies $\vol(T(\Sigma,r))\leq F(\vol(\Sigma),\norm{(\rho_k - H)_\m}_p,r)$.  

Given a closed minimal submanifold $\Sigma^m$ of an $n$-dimensional closed Riemannian manifold satisfying $\vol(M)\geq v_0$ and $\textup{diam}(M)\leq D$, since $M\subset T(\Sigma,D)$ we have
$$v_0\leq \vol(M) = \vol(T(\Sigma,D)) \leq F(\vol(\Sigma),\norm{(\rho_k-H)_\m}_p,D).$$
Since $v_0$ is fixed, for sufficiently small $\epsilon$ there exists a number $\delta>0$ such that if $\norm{(\rho_k-H)_\m}_p\leq \epsilon$ then $\vol(\Sigma)\geq\delta$.
\end{proof}

\bibliographystyle{amsplain}
\bibliography{chahine2018}

\end{document}